\documentclass[11pt]{amsart}
\usepackage[margin=1in]{geometry}

\usepackage{amssymb}
\usepackage{amsthm}
\usepackage{amsmath}
\usepackage{mathrsfs}
\usepackage{amsbsy}
\usepackage[all]{xy}
\usepackage{bm}
\usepackage{hyperref}
\usepackage{tikz}
\usepackage{array}
\usepackage{float}
\usepackage{enumerate}
\usepackage{xcolor}
\usepackage{hhline}
\allowdisplaybreaks

\newtheorem{theorem}{Theorem}[section]
\newtheorem{proposition}[theorem]{Proposition}
\newtheorem{corollary}[theorem]{Corollary}
\newtheorem{conjecture}[theorem]{Conjecture}
\newtheorem{question}[theorem]{Question}
\newtheorem{lemma}[theorem]{Lemma}

\theoremstyle{definition}
\newtheorem{definition}{Definition}[section]
\newtheorem{remark}{Remark}[section]

\DeclareMathOperator{\Id}{Id}
\DeclareMathOperator{\fast}{\mathsf{hare}}
\DeclareMathOperator{\slow}{\mathsf{tortoise}}
\DeclareMathOperator{\Image}{Im}

\begin{document}

\title[]{Stack-sorting for words} 

\author[]{Colin Defant}
\address[]{Fine Hall, 304 Washington Rd., Princeton, NJ 08544}
\email{cdefant@princeton.edu}
\author[]{Noah Kravitz}
\address[]{Grace Hopper College, Yale University, New Haven, CT 06510, USA}
\email{noah.kravitz@yale.edu}

\begin{abstract} 
We introduce operators $\fast$ and $\slow$, which act on words as natural generalizations of West's stack-sorting map. We show that the heuristically slower algorithm $\slow$ can sort words arbitrarily faster than its counterpart $\fast$. We then generalize the combinatorial objects known as valid hook configurations in order to find a method for computing the number of preimages of any word under these two operators. We relate the question of determining which words are sortable by $\fast$ and $\slow$ to more classical problems in pattern avoidance, and we derive a recurrence for the number of words with a fixed number of copies of each letter (permutations of a multiset) that are sortable by each map. In particular, we use generating trees to prove that the $\ell$-uniform
words on the alphabet $[n]$ that avoid the patterns $231$ and $221$ are counted by the $(\ell+1)$-Catalan number $\frac{1}{\ell n+1}{(\ell+1)n\choose n}$. We conclude with several open problems and conjectures. 

 \end{abstract}
\maketitle

\section{Introduction}
\subsection{Background}

Throughout this paper, the term \emph{word} refers to a finite string of letters taken from the alphabet $\mathbb N$ of positive integers. Given a word $p=p_1\cdots p_k$, we say a word $w=w_1\cdots w_n$ \emph{contains the pattern} $p$ is there are indices $i_1<\cdots<i_k$ such that $w_{i_1}\cdots w_{i_k}$ has the same relative order as $p$. Otherwise, we say $w$ \emph{avoids} $p$. For example, $3422155$ contains the pattern $211$ because the letters $322$ have the same relative order in $w$ as $211$. On the other hand, $3422155$ avoids the pattern $4321$. 

A \emph{permutation} is a word in which no letter appears more than once; it is in the context of permutations that pattern avoidance has received the most attention. Let $S_n$ denote the set of permutations whose letters are the elements of the set $\{1,\ldots,n\}$. The study of pattern avoidance in permutations originated in Knuth's monograph \emph{The Art of Computer Programming} \cite{Knuth}. Knuth defined a sorting algorithm that makes use of a vertical \emph{stack}, and he showed that this algorithm sorts a permutation into increasing order if and only if it avoids the pattern $231$. In his 1990 Ph.D. thesis, West \cite{West} introduced a deterministic variant of Knuth's algorithm, which we call the \emph{stack-sorting map} and denote by $s$. This map operates as follows. 

Place the input permutation on the right side of a vertical ``stack." At each point in time, if the stack is empty or the leftmost entry on the right side of the stack is smaller than the entry at the top of the stack, \emph{push} that leftmost entry into the stack. If there is no entry on the right of the stack or if the leftmost entry on the right side of the stack is larger than the entry on the top of the stack, \emph{pop} the top entry out of the stack and add it to the end of the growing output permutation to the left of the stack. Let $s(\pi)$ denote the output permutation that is obtained by sending $\pi$ through the stack. 
Figure \ref{Fig1} illustrates this procedure for $s(4162)=1426$.  

\begin{figure}[H]
\begin{center}
\includegraphics[width=1\linewidth]{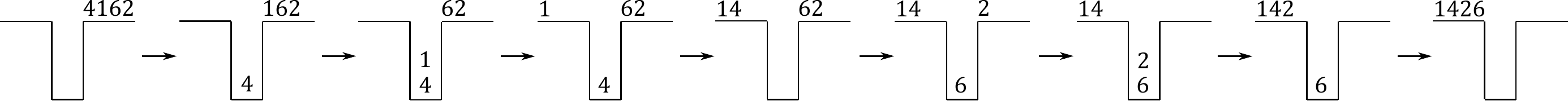}
\end{center}  
\caption{The stack-sorting map $s$ sends $4162$ to $1426$.}\label{Fig1}
\end{figure}

If $\pi$ is a permutation with largest entry $n$, we can write $\pi=LnR$, where $L$ (respectively, $R$) is the (possibly empty) substring of $\pi$ to the left (respectively, right) of the entry $n$. West observed that the stack-sorting map can be defined recursively by $s(\pi)=s(L)s(R)n$. It is also possible to define the map $s$ in terms of tree traversals of decreasing binary plane trees; we will revisit this idea in Section 3. 

We do not attempt to give a comprehensive treatment of the extensive literature concerning the stack-sorting map $s$. Instead, we provide the background that is immediately relevant to our investigations and refer the interested reader to \cite{Bona, BonaSurvey, Defant2,Defant5} (and the references therein) for further information. 

A permutation $\pi\in S_n$ is called \emph{$t$-stack-sortable} if $s^t(\pi)=123\cdots n$, where $s^t$ denotes the composition of $s$ with itself $t$ times. A $1$-stack-sortable permutation is simply called \emph{sortable}. It follows from Knuth's work that a permutation is sortable if and only if it avoids the pattern $231$. According to the well-known enumeration of $231$-avoiding permutations, there are $C_n$ sortable permutations in $S_n$,
where $C_n=\frac{1}{n+1}{2n\choose n}$ is the $n$-th Catalan number. West \cite{West} conjectured that there are exactly \[\frac{2}{(n+1)(2n+1)}{3n\choose n}\] $2$-stack-sortable permutations in $S_n$, and Zeilberger \cite{Consumption} later proved this fact. 

Much of the study of the stack-sorting map can be phrased in terms of preimages of permutations under $s$. In fact, the study of stack-sorting preimages of permutations dates back to West \cite{West}, who called $|s^{-1}(\pi)|$ the \emph{fertility} of the permutation $\pi$ and computed this fertility for some specific types of permutations. Bousquet-M\'elou \cite{Bousquet} later studied permutations with positive fertilities, which she called \emph{sorted} permutations. In doing so, she asked for a method for computing the fertility of any given permutation. The first author achieved this in much greater generality in \cite{Defant} and \cite{Defant2} by developing a theory of new combinatorial objects called \emph{valid hook configurations}. The authors of \cite{Defant3} have used valid hook configurations to find connections among permutations with fertility $1$, certain weighted set partitions, and cumulants arising in free probability theory. The first author has investigated which numbers arise as the fertilities of permutations \cite{Defant4}. In studying preimages of permutation classes under the stack-sorting map, he has also obtained several enumerative results that link the stack-sorting map with well-studied sequences \cite{Defant5}. 

Several authors have extended the well-studied area of pattern avoidance in permutations to pattern avoidance in words \cite{Albert,Atkinson,Branden,
Burstein,Burstein2,Heubach,
Heubach2,Mansour,Pudwell}. One motivation for this line of inquiry comes from the study of sorting algorithms defined on words \cite{Albert,Atkinson}. The first order of business in this article is to extend West's stack-sorting map $s$ so that it can operate on words. There is one point of ambiguity in how one defines this extension: should a letter be allowed to sit on top of a copy of itself in the stack? If, for instance, we send the word $221$ through the stack, we want to know if the second $2$ forces the first $2$ to pop out of the stack. Depending on which convention is used, the output permutation could either be $122$ or $212$; we avoid this potential issue by considering both variations.  With this background in mind, we offer the following recursive definitions of the functions $\fast$ and $\slow$ from the set of all words to itself.

\begin{definition}
First, let $\fast(\varepsilon)=\slow(\varepsilon)=\varepsilon$, where $\varepsilon$ is the empty word. Now, suppose $w$ is a nonempty word with largest letter $n$. If the letter $n$ appears $k$ times in $w$, then we can uniquely write $w=A_1nA_2n\cdots nA_{k+1}$, where the letters in the (possibly empty) words $A_i$ are all at most $n-1$. We now define
$$\fast(w)=\fast(A_1)\fast(A_2)\cdots\fast(A_{k+1})nn\cdots n$$ and 
$$\slow(w)=\slow(A_1)\slow(A_2)n\slow(A_3)n\cdots n\slow(A_k)n\slow(A_{k+1})n,$$ where there are exactly $k$ copies of the letter $n$ at the end of the word $\fast(w)$.
\end{definition}

The map $\fast$ operates by sending a word through the stack with the convention that a letter \emph{can} sit on top of a copy of itself in the stack. On the other hand, $\slow$ operates by sending a word through the stack with the convention that a letter \emph{cannot} sit on top of a copy of itself. The main purposes of this article are to compare the functions $\fast$ and $\slow$ and to show how many of the properties of West's stack-sorting map $s$ extend to this new setting. In particular, we develop a method for computing the number of preimages of a given word under each map.

\subsection{Notation}

We require the following notation in order to state our main results.  
\begin{itemize}
\item Let $\mathcal W$ denote the set of all words of finite length over the alphabet $\mathbb N$. This set is a monoid with concatenation as its binary operation. As such, $A_1 \cdots A_k$ denotes the concatenation of the words $A_1,\ldots,A_k$. We will often speak of a word $w=w_1\cdots w_m$; unless otherwise stated, $w_1,\ldots,w_m$ are assumed to be the \emph{letters} of the word $w$ (so $w$ has length $m$). 
\item Given a tuple $\textbf{c}=(c_1, \dots, c_n)$ of nonnegative integers, let $\mathcal{W}_{\textbf{c}}$ be the set of all words with exactly $c_i$ $i$'s for each $1\leq i \leq n$. One can think of $\mathcal W_{\textbf{c}}$ as the set of permutations of the multiset $\{1^{c_1},2^{c_2},\ldots,m^{c_m}\}$. 
\item Let $\Id_{\mathbf{c}}$ be the unique word in $\mathcal{W}_{\mathbf{c}}$ whose letters are nondecreasing from left to right. By abuse of terminology, we call $\Id_{\mathbf{c}}$ the \textit{identity word} in $\mathcal W_\textbf{c}$.  We will omit the subscript when it is obvious from context.
\item We call a word \emph{normalized} if it is an element of $\mathcal W_{\textbf c}$ for some vector $\textbf{c}=(c_1,\ldots,c_n)$ in which each $c_i$ is strictly positive. For example, $31341$ is not normalized because it does not contain the letter $2$. 
\item Let $\fast^k$ denote the map $\fast$ composed with itself $k$ times, and define $\slow^k$ similarly. Given a word $w \in \mathcal{W}_{\textbf{c}}$, let $\langle w \rangle_{\textsf{fast}}$ be the smallest nonnegative integer $k$ such that $\textsf{fast}^k(w)=\Id_{\textbf{c}}$.  Similarly, let $\langle w \rangle_{\textsf{slow}}$ be the smallest nonnegative integer $k$ such that $\textsf{slow}^k(w)=\Id_{\textbf{c}}$.  In particular, put $\langle \varepsilon \rangle_{\fast}=\langle \varepsilon \rangle_{\slow}=0$. These values measure how ``far" $w$ is from the identity word under our generalized stack-sorting maps.
\item A \emph{composition} of a positive integer $m$ is a tuple of positive integers that sum to $m$.
\end{itemize}

\subsection{Outline of the Paper}
The operators $\fast$ and $\slow$ get their names from the heuristic idea that iteratively applying the map $\fast$ to a word should produce an identity word at least as fast as iteratively applying $\slow$ does. More formally, it seems reasonable to expect that $\langle w\rangle_{\fast}\leq\langle w\rangle_{\slow}$ for every word $w$. For example, $\langle 2221\rangle_{\fast}=1<3=\langle 2221\rangle_{\slow}$. However, in some special cases, we find the fable had it right: slow and steady wins the race! That is, there exist words $w$ for which $\langle w\rangle_{\fast}>\langle w\rangle_{\slow}$. In Section $2$, we will construct a word $\eta_n$ of length $2n+1$ such that $\langle \eta_n\rangle_{\fast}=2n-2$ and $\langle \eta_n\rangle_{\slow}=n$ for each positive integer $n$. In the same section, we show how to rewrite these maps in terms of West's stack-sorting map $s$ and also analyze the worst-case-scenario sorting for each map. 

In Section 3, we describe the aforementioned connection between $s$ and tree traversals of decreasing binary plane trees. We then explain the analogous connection for the maps $\fast$ and $\slow$. Answering a question raised in \cite{Defant3}, we generalize valid hook configurations to words, and we use these objects to show how to calculate the number of preimages of a word under the maps $\fast$ and $\slow$. This vastly generalizes the work on computing fertilities of permutations undertaken by West \cite{West} and the first author \cite{Defant,Defant2}.  

In Section 4, we utilize the ideas developed in Section 3 to study what we call $\fast$-fertility numbers and $\slow$-fertility numbers. Specifically, we show that for every nonnegative integer $f$, there exists a word $w$ such that $|\fast^{-1}(w)|=|\slow^{-1}(w)|=f$. As demonstrated in \cite{Defant4}, this result is false if we require our words to be permutations.

In Section 5, we show that a word $w\in\mathcal W_{\textbf c}$ satisfies $\fast(w)=\Id_{\textbf c}$ if and only if it avoids the pattern $231$ and satisfies $\slow(w)=\Id_{\textbf c}$ if and only if it avoids the patterns $231$ and $221$. We discuss known results concerning words that avoid the pattern $231$ and present new enumerative results concerning words that avoid the patterns $231$ and $221$. Specifically, we provide a recurrence for $\mathcal N(c_1,\ldots,c_n)$, the number of words in $\mathcal W_{(c_1,\ldots,c_n)}$ that avoid $231$ and $221$. We also use generating trees to prove that \[\mathcal{N}(\underbrace{\ell, \ell, \dots, \ell}_{n})=\frac{1}{\ell n+1}{(\ell+1)n\choose n}.\] In Section 6, we list several open problems and conjectures. 

\section{The Tortoise and the Hare}
We begin this section by recasting $\fast$ and $\slow$ explicitly in terms of the action of West's stack-sorting map $s$.  Given a vector $\textbf{c}=(c_1, \dots, c_n)$ of nonnegative integers, define the maps $\phi_{\textbf{c}}^{asc}, \phi_{\textbf{c}}^{des}: \mathcal{W}_{\textbf{c}} \to S_{c_1+\cdots +c_n}$ (recall that $S_m$ is the set of permutations of $\{1,\ldots,m\}$) as follows.  For each $1\leq i \leq n$, let $p_i=c_1+\cdots +c_{i-1}$.  To obtain $\phi_{\textbf{c}}^{asc}(w)$ from $w$, we replace the $i$'s by the integers $p_i+1, p_i+2, \dots, p_i+c_i$ in ascending order for each $i$.  To obtain $\phi_{\textbf{c}}^{des}(w)$, we replace the $i$'s by the integers $p_i+1, p_i+2, \dots, p_i+c_i$ in descending order for each $i$.  Note that even though these maps are not surjective if any $c_i>1$, they are always injective.  We define the map $\psi_{\textbf{c}}:S_{c_1+\cdots +c_n} \to \mathcal{W}_{\textbf{c}}$ as follows.  To obtain $\psi_{\textbf{c}}(\pi)$ from $\pi$, we replace all of the digits $p_i+1, p_i+2, \dots, p_i+c_i$ by the letter $i$ for each $1 \leq i \leq n$.  Clearly, $\psi_{\textbf{c}} \circ \phi_{\textbf{c}}^{asc}=\psi_{\textbf{c}} \circ \phi_{\textbf{c}}^{des}: \mathcal{W}_{\textbf{c}} \to \mathcal{W}_{\textbf{c}}$ is the identity map.  Similarly, $\phi_{\textbf{c}}^{asc} \circ \psi_{\textbf{c}}: \Image(\phi_{\textbf{c}}^{asc}) \to \Image(\phi_{\textbf{c}}^{asc})$ and $\phi_{\textbf{c}}^{des} \circ \psi_{\textbf{c}}: \Image(\phi_{\textbf{c}}^{des}) \to \Image(\phi_{\textbf{c}}^{des})$ are both the identity map (restricted to the correct subset of $S_{c_1+\cdots +c_n}$). Consequently, $\psi_{\textbf{c}}$ is a left inverse for both $\phi_{\textbf{c}}^{asc}$ and $\phi_{\textbf{c}}^{des}$.

As an example, $\phi_{(2,2,3)}^{asc}(3313221)=5617342$ and $\phi_{(2,2,3)}^{des}(3313221)=7625431$.  We emphasize that $\psi_{\textbf{c}}$ depends strongly on $\textbf{c}$. For example, $\psi_{(2,2,3)}(5617342)=3313221$ as expected, whereas $\psi_{(2,3,2)}(5617342)=2313221$ and $\psi_{(6,1)}(5617342)=1112111$.  The following lemma reduces the computation of $\fast$ and $\slow$ to computations involving $s$.

\begin{proposition}\label{prop:permexpand}
For every word $w \in \mathcal{W}_{\textbf{c}}$, we have
$$\fast(w)=(\psi_{\textbf{c}} \circ s \circ \phi_{\textbf{c}}^{des})(w) \quad \text{and} \quad \slow(w)=(\psi_{\textbf{c}} \circ s \circ \phi_{\textbf{c}}^{asc})(w).$$  Moreover, for every positive integer $k$, we have $$\slow^k(w)=(\psi_{\textbf{c}} \circ s^k \circ \phi_{\textbf{c}}^{asc})(w).$$
\end{proposition}

\begin{proof}
Fix a word $w \in \mathcal{W}_{\textbf{c}}$.  For the first statement, consider the permutation $\phi_{\textbf{c}}^{des}(w)$.  We may associate each entry of $\phi_{\textbf{c}}^{des}(w)$ with the letter that appears in the corresponding position in $w$.  If we keep track of the positions of individual entries and letters when we apply $s$ to $\phi_{\textbf{c}}^{des}(w)$ and $\fast$ to $w$, we see that the corresponding entries and letters enter the stack and pop out of the stack identically.  Hence, when we apply $\psi_{\textbf{c}}$ to $s(\phi_{\textbf c}^{des}(w))$, each entry is taken to the correct letter value in $\fast(w)$.  This shows that $\fast(w)=(\psi_{\textbf{c}} \circ s \circ \phi_{\textbf{c}}^{des})(w)$.  The same argument shows that $\slow(w)=(\psi_{\textbf{c}} \circ s \circ \phi_{\textbf{c}}^{asc})(w)$.

For the second statement, it suffices to note that $s$ maps $\Image(\phi_{\textbf{c}}^{asc})$ into itself.\footnote{It is not difficult to see that $(s \circ \phi_{\textbf{c}}^{des})(w)\not\in\Image(\phi_{\textbf{c}}^{des})$ if two letters of $w$ with the same value are ever in the stack simultaneously during the $\fast$-sorting process.} This follows from the simple observation that if $a<b$ and $a$ appears before $b$ in a permutation $\pi$, then $a$ appears before $b$ in $s(\pi)$. 
\end{proof}

The maps $\fast$ and $\slow$ do in fact ``sort'' words in the sense that iterative applications of either map to any word will eventually reach an identity word, which is a fixed point.  A natural question is how many iterations it takes to reach this fixed point.  Recall that $\langle \cdot \rangle_{\fast}$ and $\langle \cdot \rangle_{\slow}$ measure this ``distance'' from the identity.  In each $\mathcal{W}_{\mathbf{c}}$, this metric equals $0$ for only the identity word, and it equals $1$ for the nonidentity words that are completely sorted in a single go.  

Intuitively, $\fast$ should be the more efficient sorting algorithm because a later occurrence of a large letter value does not cause the previous occurrences to pop out of the stack prematurely.  It is easy to show that worst-case-scenario sorting with $\fast$ is much more efficient than worst-case-scenario sorting with $\slow$. For instance, if $w$ is a word with largest letter $n$, then all of the $n$'s are at the very end of $\fast(w)$, whereas only one $n$ is guaranteed to be at the end of $\slow(w)$.  In fact, this ``rate of progress'' is the worst-case scenario for each map; this is a natural way in which $\fast$ is faster than $\slow$.

\begin{proposition}
Let $\textbf{c}=(c_1, \dots, c_n)$, where $c_1,\ldots,c_n$ are positive integers.  For every $w \in \mathcal{W}_{\textbf{c}}$, we have $$\langle w \rangle_{\fast} \leq n-1 \quad \text{and} \quad \langle w \rangle_{\slow} \leq c_2+c_3+\cdots+c_n.$$  Moreover, equality is achieved in both cases by the word $\rho \in \mathcal{W}_{\textbf{c}}$ that is obtained from $\Id_{\textbf{c}}$ by moving all of the $1$'s to the end of the word.
\end{proposition}

\begin{proof}
Fix some $w \in \mathcal{W}_{\textbf{c}}$. For the sake of clarity, let $i^{c_i}$ denote the word formed by concatenating the letter $i$ with itself $c_i$ times. It is clear that $n^{c_n}$ appears at the very end of $\fast(w)$.  By induction, we see that for every $1 \leq k \leq n-1$, the word $\fast^k(w)$ ends with the string \[(n-k+1)^{c_{n-k+1}}(n-k+2)^{c_{n-k+2}}\cdots(n-1)^{c_{n-1}}n^{c_n}.\]  In particular, $\fast^{n-1}(w)=\Id_{\textbf c}$, which establishes the first inequality.  In much the same way, we know that $\slow^k(w)$ ends with the $k$ largest letters in increasing order.  This implies that $\slow^{c_2+c_3+\cdots+c_n}(w)=\Id_{\textbf c}$ and establishes the second inequality.

We now prove the second part of the lemma. By definition, $\rho=2^{c_2} 3^{c_3}\cdots n^{c_n} 1^{c_1}$.  Induction on $k$ shows that $\fast^k(\rho)=2^{c_2} 3^{c_3} \cdots (n-k)^{c_{n-k}} 1^{c_1} (n-k+1)^{c_{n-k+1}} \cdots n^{c_n}$ for each $0 \leq k \leq n-1$.  Hence, $\langle \rho \rangle_{\fast}=n-1$.  Similarly, each iterative application of $\slow$ to $\rho$ moves the letter directly to the left of the $1$'s to the position directly to the right of the $1$'s (which stay together).  Hence, $\langle \rho \rangle_{\slow}=c_2+\cdots +c_n$.
\end{proof}

In light of the previous lemma, one would na\"{i}vely expect $\fast$ to sort all words faster than $\slow$, i.e., $\langle w \rangle_{\fast}\leq \langle w \rangle_{\slow}$.  However, this turns out not always to happen: even though $\fast$ seems to make more progress in the first few iterations, sometimes $\slow$ catches up and reaches the identity first! For example, we have \[3662451\xrightarrow{\fast}3241566\xrightarrow{\fast}2314566\xrightarrow{\fast}2134566\xrightarrow{\fast}1234566\] and \[3662451\xrightarrow{\slow}3624156\xrightarrow{\slow}3214566\xrightarrow{\slow}1234566,\]
so \[\langle 3662451\rangle_{\fast}=4>3=\langle 3662451\rangle_{\slow}.\]
The following theorem shows that $\slow$ can actually be arbitrarily faster than $\fast$.

\begin{theorem}\label{thm:tortoisebeatsharebyalot}
For any integer $n \geq 3$, the word \[\eta_n=357\cdots(2n-3)(2n)(2n)246\cdots(2n-2)(2n-1)1\] has length $2n+1$ and satisfies \[\langle \eta_n\rangle_{\fast}=2n-2\quad\text{and}\quad\langle \eta_n\rangle_{\slow}=n.\]
\end{theorem}

\begin{proof}
The proof of the theorem amounts to observing what happens to $\eta_n$ under repeated applications of $\fast$ and $\slow$. One could write out these calculations for general $n$, but we fear that doing so would only obfuscate the computations with a sea of ellipses ($\cdots$). Instead, we show the calculations for the case $n=5$; the general case is completely analogous. 

We have $\eta_5=3\,\,5\,\,7\,\,10\,\,10\,\,2\,\,4\,\,6\,\,8\,\,9\,\,1$. Now, 
\begin{alignat*}{2}
&3\,\,5\,\,7\,\,10\,\,10\,\,2\,\,4\,\,6\,\,8\,\,9\,\,1\hspace{2cm}&&3\,\,5\,\,7\,\,10\,\,10\,\,2\,\,4\,\,6\,\,8\,\,9\,\,1 \\
 &\hspace{1.75cm}\Big\downarrow\fast &&\hspace{1.75cm}\Big\downarrow\slow \\
&3\,\,5\,\,7\,\,2\,\,4\,\,6\,\,8\,\,1\,\,9\,\,10\,\,10\hspace{2cm}&&3\,\,5\,\,7\,\,10\,\,2\,\,4\,\,6\,\,8\,\,1\,\,9\,\,10 \\
 &\hspace{1.75cm}\Big\downarrow\fast &&\hspace{1.75cm}\Big\downarrow\slow \\
 &3\,\,5\,\,2\,\,4\,\,6\,\,7\,\,1\,\,8\,\,9\,\,10\,\,10\hspace{2cm}&&3\,\,5\,\,7\,\,2\,\,4\,\,6\,\,1\,\,8\,\,9\,\,10\,\,10 \\
 &\hspace{1.75cm}\Big\downarrow\fast &&\hspace{1.75cm}\Big\downarrow\slow \\
&3\,\,2\,\,4\,\,5\,\,6\,\,1\,\,7\,\,8\,\,9\,\,10\,\,10\hspace{2cm}&&3\,\,5\,\,2\,\,4\,\,1\,\,6\,\,7\,\,8\,\,9\,\,10\,\,10 \\
 &\hspace{1.75cm}\Big\downarrow\fast &&\hspace{1.75cm}\Big\downarrow\slow \\
 &2\,\,3\,\,4\,\,5\,\,1\,\,6\,\,7\,\,8\,\,9\,\,10\,\,10\hspace{2cm}&&3\,\,2\,\,1\,\,4\,\,5\,\,6\,\,7\,\,8\,\,9\,\,10\,\,10 \\
 &\hspace{1.75cm}\Big\downarrow\fast &&\hspace{1.75cm}\Big\downarrow\slow \\
&2\,\,3\,\,4\,\,1\,\,5\,\,6\,\,7\,\,8\,\,9\,\,10\,\,10\hspace{2cm}&&1\,\,2\,\,3\,\,4\,\,5\,\,6\,\,7\,\,8\,\,9\,\,10\,\,10 \\
 &\hspace{1.75cm}\Big\downarrow\fast &&\hspace{1.75cm} \\
&2\,\,3\,\,1\,\,4\,\,5\,\,6\,\,7\,\,8\,\,9\,\,10\,\,10\hspace{2cm}&& \\
 &\hspace{1.75cm}\Big\downarrow\fast &&\hspace{1.75cm} \\
&2\,\,1\,\,3\,\,4\,\,5\,\,6\,\,7\,\,8\,\,9\,\,10\,\,10\hspace{2cm}&& \\
 &\hspace{1.75cm}\Big\downarrow\fast &&\hspace{1.75cm}
\\
&1\,\,2\,\,3\,\,4\,\,5\,\,6\,\,7\,\,8\,\,9\,\,10\,\,10\hspace{2cm}&& \hspace{4cm}. 
 \end{alignat*}
\end{proof}

Say a word $w$ is \emph{exceptional} if $\langle w\rangle_{\fast}>\langle w\rangle_{\slow}$. Let $\mathcal E_m$ be the set of exceptional normalized words of length $m$. It turns out that $\mathcal E_m=\emptyset$ when $m\leq 6$. We have used a computer to find that \[\mathcal E_7=\{3662451,3664251,6362451,6364251\}.\] The sets $\mathcal E_8$ and $\mathcal E_9$ have $172$ and $5001$ elements, respectively. Furthermore, we have checked that each element of $\mathcal E_8$ contains one of the words in $\mathcal E_7$ as a pattern. We have also found that there are $72$ words $w$ of length $9$ (but no shorter words) that satisfy $\langle w\rangle_{\fast}=\langle w\rangle_{\slow}+2$. These observations lead to a host of questions concerning exceptional words, many of which we list in Section $6$.

\section{Trees and Valid Hook Configurations}
Given a set $X$ of positive integers, a \emph{decreasing plane tree on $X$} is a rooted plane tree in which the vertices are labeled with the elements of $X$ (where each label is used exactly once) such that each nonroot vertex has a label that is strictly smaller than the label of its parent.
A \emph{binary plane tree} either is empty or consists of a root vertex along with an ordered pair of subtrees (the left and right subtrees) that are themselves binary plane trees. Note that if a vertex in a binary plane tree has a single child, we make a distinction between whether the child is a left child or a right child. Figure \ref{Fig2} shows two different decreasing binary plane trees on $\{1,2,\ldots,7\}$.  

\begin{figure}[h]
\begin{center} 
\includegraphics[height=2cm]{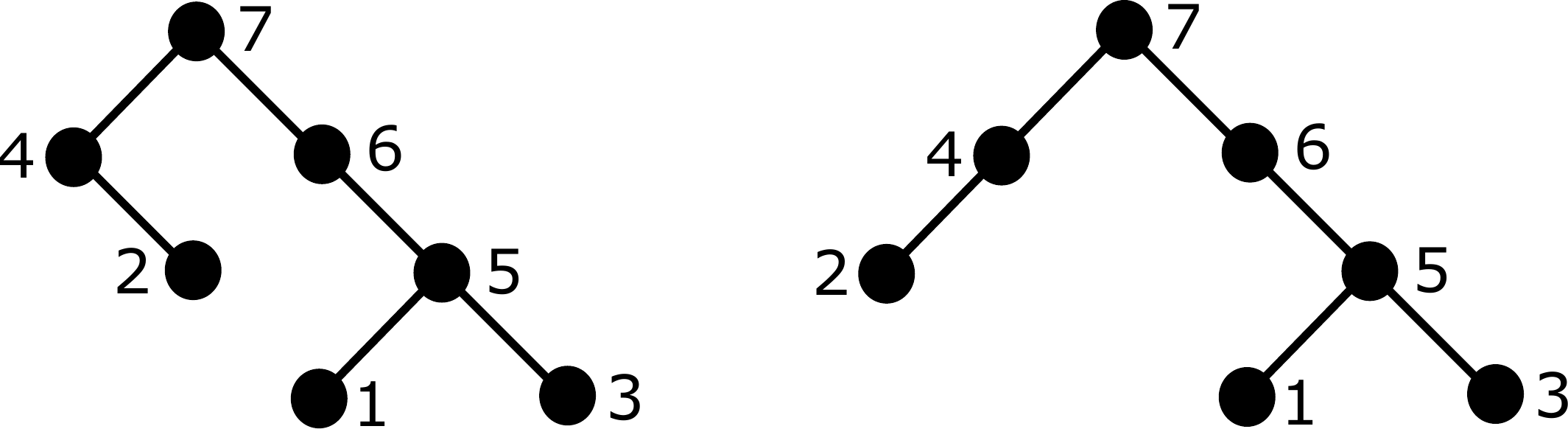}
\end{center}
\caption{Two different decreasing binary plane trees on $\{1,2,\ldots,7\}$.}\label{Fig2}
\end{figure}

We can use a \emph{tree traversal} to read the labels of a decreasing binary plane tree. One tree traversal, called the \emph{in-order reading} (sometimes called the \emph{symmetric order reading}), is obtained by reading the left subtree of the root in in-order, then reading the label of the root, and finally reading the right subtree of the root in in-order. Let $S(\tau)$ denote the in-order reading of a decreasing binary plane tree $\tau$. It turns out that $S$ gives a bijection from the set of decreasing binary plane trees on a set $X$ to the set of permutations of $X$ \cite{Stanley}. Under this bijection, the tree on the left in Figure \ref{Fig2} corresponds to the permutation $4276153$, and the tree on the right corresponds to the permutation $2476153$. 

Another tree traversal, called the \emph{postorder reading}, is defined for arbitrary decreasing plane trees. We read a decreasing plane tree in postorder by reading the subtrees of the root from left to right (each in postorder) and then reading the label of the root. For example, both trees in Figure \ref{Fig2} have postorder $2413567$. Let $P(\tau)$ denote the postorder reading of a decreasing plane tree $\tau$. The fundamental result \cite[Corollary 8.22]{Bona} that links West's stack-sorting map to decreasing binary plane trees is the fact that \[s=P\circ S^{-1}.\] In other words, if we are given a permutation $\pi$, then $s(\pi)$ is the postorder reading of the unique decreasing binary plane tree whose in-order reading is $\pi$. 

We can generalize the notion of a decreasing plane tree to that of a weakly decreasing plane tree. If $\mathcal X$ is a multiset of positive integers, then a \emph{weakly decreasing plane tree on $\mathcal X$} is a rooted plane tree labeled with the elements of $\mathcal X$ (where each label appears exactly as many times as it appears in $\mathcal X$) such that each nonroot vertex has a label that is at most as large as the label of its parent. The definitions of in-order and postorder readings extend in the obvious ways to weakly decreasing binary plane trees. As before, let $S(\tau)$ and $P(\tau)$ denote the in-order and postorder readings, respectively, of a weakly decreasing (binary) plane tree $\tau$. Let $\mathcal L$ be the set of weakly decreasing binary plane trees in which a vertex and its right child cannot have the same label (i.e., whenever a vertex has the same label as its parent, it must be a left child). Similarly, let $\mathcal R$ be the set of weakly decreasing binary plane trees in which a vertex cannot have the same label as its left child. 

The bijection between permutations and decreasing binary plane trees extends to the context of words. More precisely, the in-order reading $S$ provides a bijection from $\mathcal L$ to the set $\mathcal W$ of all words (that is, the set  of finite words over $\mathbb N$). Similarly, $S$ provides a bijection from $\mathcal R$ to $\mathcal W$. To be completely formal, we write \[S_{\mathcal L}:\mathcal L\to\mathcal{W} \quad\text{and}\quad S_{\mathcal R}:\mathcal R\to\mathcal W\]
for these bijections. The inverse maps $S_{\mathcal R}^{-1}$ and $S_{\mathcal L}^{-1}$ are defined recursively as follows. Given a word $w$, we can write $w=AnB$, where $n$ is the largest letter in $w$ and $A$ does not contain the letter $n$ (so that all copies of the letter $n$ in $w$ appear in the subword $nB$). We then let $S_{\mathcal R}^{-1}(w)$ be the tree in which the root vertex has label $n$ and in which the left and right subtrees of the root are $S_{\mathcal R}^{-1}(A)$ and $S_{\mathcal R}^{-1}(B)$, respectively. Similarly, we can write $w=CnD$, where $n$ is the largest letter in $w$ and $D$ does not contain the letter $n$ (so that all copies of the letter $n$ in $w$ appear in the subword $Cn$). Let $S_{\mathcal L}^{-1}(w)$ be the tree in which the root vertex has label $n$ and in which the left and right subtrees of the root are $S_{\mathcal L}^{-1}(C)$ and $S_{\mathcal L}^{-1}(D)$, respectively. We omit the straightforward proof that these maps are in fact inverses of $S_{\mathcal R}$ and $S_{\mathcal L}$, respectively. Figure \ref{Fig3} show the trees $S_{\mathcal R}^{-1}(w)$ and $S_{\mathcal L}^{-1}(w)$ when $w=23123311$.

\begin{figure}[t]
\begin{center} 
\includegraphics[height=2.5cm]{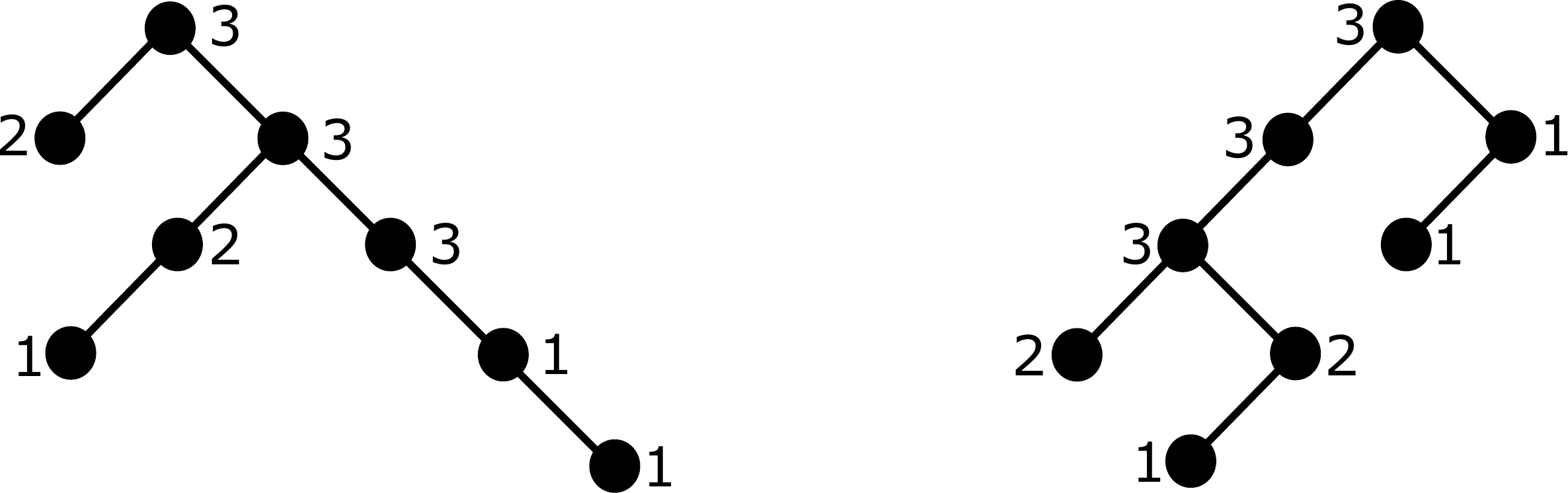}
\end{center}
\caption{The trees $S_{\mathcal R}^{-1}(23123311)$ (left) and $S_{\mathcal L}^{-1}(23123311)$ (right).}\label{Fig3}
\end{figure}

The motivation for defining the sets $\mathcal L$ and $\mathcal R$ comes from the following lemma.

\begin{lemma}\label{lem:treecorrespondence}
The maps $\fast$ and $\slow$ satisfy \[\fast=P \circ S_{\mathcal R}^{-1}\quad\text{and}\quad\slow=P \circ S_{\mathcal L}^{-1}.\]
\end{lemma}

\begin{proof}
This is immediate from the definitions of $S_{\mathcal R}^{-1}$, $S_{\mathcal L}^{-1}$, $P$, $\fast$, and $\slow$.  
\end{proof}

As mentioned in the introduction, much of the theory of the stack-sorting map $s$ can be phrased in terms of preimages of permutations. Lemma \ref{lem:treecorrespondence} provides a link between weakly decreasing binary plane trees and the maps $\fast$ and $\slow$; it is this link that allows us to study preimages of permutations under $\fast$ and $\slow$. Because $S_{\mathcal R}:\mathcal R\to\mathcal W$ is a bijection, it follows from Lemma \ref{lem:treecorrespondence} that $|\fast^{-1}(w)|$ is the number of trees in $\mathcal R$ with postorder $w$. Similarly, $|\slow^{-1}(w)|$ is the number of trees in $\mathcal L$ with postorder $w$. 

In \cite{Defant}, the first author introduced new combinatorial objects called ``valid hook configurations.''  He showed how to use these objects to compute the number of decreasing plane trees of a prescribed type with a given postorder. Applying this concept to the specific case of decreasing binary plane trees, one obtains a method for computing the fertility $|s^{-1}(\pi)|$ of a permutation $\pi$. In fact, one can even count the elements of $s^{-1}(\pi)$ according to their number of descents and number of valleys. 

In the rest of this section, we discuss how to define valid hook configurations for words. Following \cite{Defant}, one could develop a general method for counting weakly decreasing plane trees of a prescribed type with a given postorder. For example, the 	``decreasing $\mathbb N$-trees" discussed in that article generalize naturally to ``weakly decreasing $\mathbb N$-trees'', and the method for counting decreasing $\mathbb N$-trees with a given postorder extends to the setting of weakly decreasing $\mathbb N$-trees with very little difficulty. Because our interest lies with the stack-sorting maps $\fast$ and $\slow$, we will not concern ourselves with very general types of trees. Instead, we will focus on counting trees in $\mathcal L$ and $\mathcal R$ with a given postorder reading. It turns out that this problem is more difficult (and hence more interesting) than the problem of counting weakly decreasing $\mathbb N$-trees. Indeed, we will see that we must place additional conditions on our valid hook configurations in order to ensure that we count trees in either $\mathcal L$ or $\mathcal R$.   

The \emph{plot} of a word $w=w_1\cdots w_m$ is the graph depicting the point $(i,w_i)$ for all $i\in\{1,\ldots,m\}$. For example, the left image in Figure \ref{Fig4} shows the plot of the word $21133$. A \emph{hook} of $w$ is drawn by selecting two points $(i,w_i)$ and $(j,w_j)$ with $i<j$ and $w_i\leq w_j$. We draw a vertical line segment from $(i,w_i)$ to $(i,w_j)$ and then connect it to a horizontal line segment from $(i,w_j)$ to $(j,w_j)$. The points $(i,w_i)$ and $(j,w_j)$ are respectively called the \emph{southwest endpoint} and the \emph{northeast endpoint} of the hook. The right image in Figure \ref{Fig4} shows two hooks drawn on the plot of $21133$. One hook has southwest endpoint $(1,2)$ and northeast endpoint $(5,3)$. The other has southwest endpoint $(4,3)$ and northeast endpoint $(5,3)$. 

\begin{figure}[t]
\begin{center}
\includegraphics[height=1.5cm]{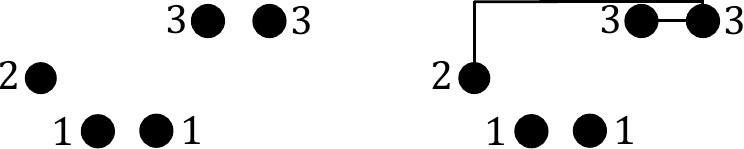}
\end{center}  
\caption{The left image shows the plot of $21133$. The right image shows this same plot along with two hooks drawn on it. For the purpose of drawing the hooks in the right image clearly, we have allowed one of the hooks to pass above the point $(4,3)$. This is simply to avoid confusion. }\label{Fig4}
\end{figure}

Consider a word $w=w_1\cdots w_m$. For our purposes, it is most convenient to define a \emph{descent} of $w$ to be an index $d\in\{1,\ldots,m-1\}$ such that $w_d\geq w_{d+1}$. 
If $d$ is a descent of $w$, we call the point $(d,w_d)$ a \emph{descent top} of $w$. Let $H$ be a hook of $w$ with southwest endpoint $(i,w_i)$ and northeast endpoint $(j,w_j)$. We say $H$ is a \emph{descent hook} if $i$ is a descent of $w$. We say $H$ is \emph{horizontal} if $w_i=w_j$, and we say $H$ is \emph{small} if $j=i+1$. For example, both of the hooks shown in the right image of Figure \ref{Fig4} are descent hooks. The hook in that figure with southwest endpoint $(4,3)$ and northeast endpoint $(5,3)$ is both horizontal and small. The other hook is neither horizontal nor small.    

Let $(H_1,\ldots,H_k)$ be a tuple of hooks of the word $w=w_1\cdots w_m$. Let $(i_u,w_{i_u})$ and $(j_u,w_{j_u})$ denote the southwest and northeast endpoints, respectively, of the hook $H_u$. We say the tuple $(H_1,\ldots,H_k)$ is a \emph{valid hook configuration} of $w$ if it satisfies the following properties: 

\begin{enumerate}[1.]
\item We have $i_1<\cdots<i_k$.
\item Each descent top of $w$ is the southwest endpoint of a hook. 
\item If $(j,w_j)$ is the northeast endpoint of any hook, then $(j,w_j)$ is both the northeast endpoint of a descent hook and the northeast endpoint of a small hook (where these hooks could be the same). 
\item For all $u,v\in\{1,\ldots,k\}$, either the intervals $(i_u,j_u)$ and $(i_v,j_v)$ are disjoint or one is contained in the other. 
\end{enumerate}

These four conditions have some immediate consequences for a valid hook configuration \linebreak $(H_1,\ldots,H_k)$ of a word $w$; they are worth keeping in mind if one wishes to work with valid hook configurations. The first consequence, which is immediate from condition 1, is that a point in the plot of $w$ can be the southwest endpoint of at most one hook. The second consequence, which follows from conditions 2 and 4, is that a hook cannot pass strictly below a point $(a,w_a)$ in the plot of $w$. The third consequence, which also follows from condition 4, is that two hooks cannot intersect each other perpendicularly except at a common endpoint. By this, we mean that the vertical part of a hook cannot intersect the horizontal part of a different hook unless the intersection occurs at a point that is a common endpoint of the two  hooks. Figure \ref{Fig5} shows three examples of these forbidden situations. Figure \ref{Fig6} depicts a valid hook configuration of the word $211232124567$.

\begin{figure}[t]
\begin{center}
\includegraphics[height=3.5cm]{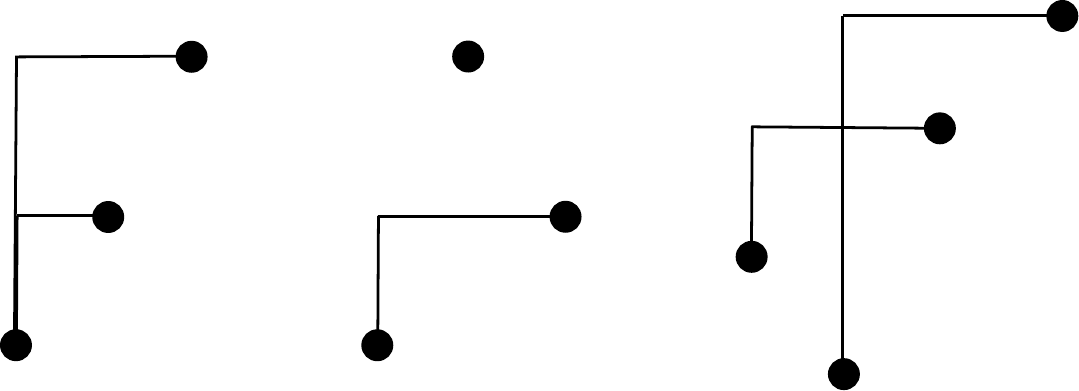}
\end{center}  
\caption{Three placements of hooks that are forbidden in a valid hook configuration.}\label{Fig5}
\end{figure}

\begin{figure}[t]
\begin{center}
\includegraphics[height=3.5cm]{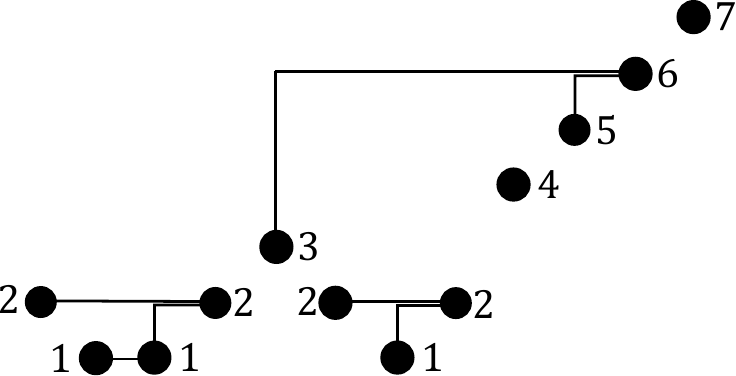}
\end{center}  
\caption{A valid hook configuration of $211232124567$.}\label{Fig6}
\end{figure}

A valid hook configuration of a word $w$ induces a coloring of the plot of $w$. To color the plot, first draw a ``sky" over the entire diagram and color the sky blue. Assign arbitrary distinct colors other than blue to the hooks in the valid hook configuration. Roughly speaking, each point should ``look up" and receive the color that it sees. If a point does not see any hooks, then it sees the sky and receives the color blue. More formally, suppose we wish to color a point $(i,w_i)$. Consider the set of all hooks that either lie above $(i,w_i)$ or have northeast endpoint $(i,w_i)$. If this set is empty, color the point $(i,w_i)$ blue. Otherwise, choose the hook from this set that has the rightmost southwest endpoint, and give $(i,w_i)$ the color of that hook. Note that we make the convention that the endpoints of a hook do not lie below that hook. Figure \ref{Fig7} shows the coloring of the plot of $211232124567$ induced by the valid hook configuration in Figure \ref{Fig6}. Note that the points $(1,2)$, $(5,3)$, and $(12,7)$ are blue because they do not lie below any hooks and are not northeast endpoints of any hooks. 

\begin{figure}[t]
\begin{center}
\includegraphics[height=4.55cm]{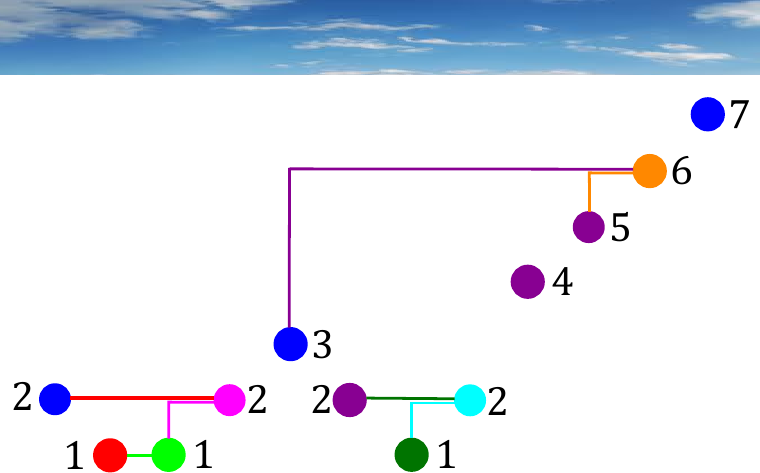}
\end{center}  
\caption{The coloring of the plot of $211232124567$ induced by the valid hook configuration of Figure \ref{Fig6}.}\label{Fig7}
\end{figure}

\begin{remark}\label{Rem1}
It is straightforward to check that a northeast endpoint of a hook cannot receive the same color as any other point. 
\end{remark}

We have shown that each valid hook configuration $\mathscr H=(H_1,\ldots,H_k)$ of a word $w=w_1\cdots w_m$ induces a coloring of the plot of $w$. From this coloring, we obtain an integer composition $q^{\mathscr H}=(q_0,\ldots,q_k)$ of $m$. For each $i>0$, we simply define $q_i$ to be the number of points that are given the same color as the hook $H_i$. We also let $q_0$ be the number of blue points in the induced coloring (i.e., the number of points that see the sky when they look up). We say the valid hook configuration $\mathscr H$ \emph{induces} the composition $q^{\mathscr H}$. A \emph{valid composition of $w$} is a composition that is induced by a valid hook configuration of $w$. 

As mentioned above, we are going to restrict our attention to counting trees in the sets $\mathcal L$ and $\mathcal R$. In order to do so, we must place additional constraints on the valid hook configurations we allow. For example, we will see later that the hooks in a valid hook configuration of a word $w$ become edges in the trees with postorder $w$. Since the trees in $\mathcal L$ and $\mathcal R$ are binary, it is natural to define a \emph{binary valid hook configuration} of a word $w$ to be a valid hook configuration of $w$ in which each northeast endpoint of a hook is the northeast endpoint of at most two hooks. For example, the valid hook configuration in Figure \ref{Fig6} is binary. Let $\mathcal H_{\mathcal R}(w)$ denote the set of binary valid hook configurations of a word $w$ in which every horizontal hook is small. Let $\mathcal H_{\mathcal L}(w)$ denote the set of all binary valid hook configurations of $w$ in which no small horizontal hook has the same northeast endpoint as another hook. 

We can finally state our main theorem connecting valid hook configurations with the problem of determining $|\fast^{-1}(w)|$ and $|\slow^{-1}(w)|$. Let $C_n=\frac{1}{n+1}{2n\choose n}$ denote the $n^\text{th}$ Catalan number. Given an integer composition $q=(q_0,\ldots,q_k)$, let \[C_q=\prod_{t=0}^kC_{q_t}.\] Recall that $q^{\mathscr H}$ denotes the composition induced by the valid hook configuration $\mathscr H$. 

\begin{theorem}\label{mainVHCthm}
For every word $w$, we have \[|\fast^{-1}(w)|=\sum_{\mathscr H\in\mathcal H_{\mathcal R}(w)}C_{q^{\mathscr H}}\quad\text{and}\quad|\slow^{-1}(w)|=\sum_{\mathscr H\in\mathcal H_{\mathcal L}(w)}C_{q^{\mathscr H}}.\]
\end{theorem}

The proof of Theorem \ref{mainVHCthm} requires the following simple yet important lemma. Informally, this lemma states that the points in each color class induced by a valid hook configuration are strictly increasing in height from left to right.  

\begin{lemma}\label{lem:colorsincrease}
Let $\mathscr H$ be a valid hook configuration of a word $w=w_1\cdots w_m$. Suppose the points $(i,w_i)$ and $(j,w_j)$ are given the same color in the coloring of the plot of $w$ induced by $\mathscr H$. If $i<j$, then $w_i<w_j$.  
\end{lemma}

\begin{proof}
Assume instead that $i<j$ and $w_i\geq w_j$. Let $i'$ be the largest integer satisfying $i'<j$ and $w_{i'}\geq w_j$. The point $(i',w_{i'})$ must be a descent top of $w$, so condition 2 in the definition of a valid hook configuration tells us that there is a hook $H$ in $\mathscr H$ whose southwest endpoint is $(i',w_{i'})$. The northeast endpoint of $H$ must either lie to the right of $(j,w_j)$ or be equal to $(j,w_j)$. The point $(j,w_j)$ must be given the same color as either $H$ or a hook whose southwest endpoint is to the right of $(i',w_{i'})$. The point $(i,w_i)$ cannot be given this color, which contradicts our hypothesis.  
\end{proof}

Let $\mathscr H=(H_1,\ldots,H_k)$ be a valid hook configuration of a word $w=w_1\cdots w_m$. Let $Q_r(\mathscr H)$ denote the set of points that are given the same color as the hook $H_r$ in the coloring induced by $\mathscr H$. Also, let $Q_0(\mathscr H)$ denote the set of blue points. For each $r\in\{0,\ldots,k\}$, Lemma \ref{lem:colorsincrease} guarantees that the heights of the points in $Q_r(\mathscr H)$ are distinct (since they are strictly increasing from left to right). Therefore, it is natural to define $X_r(\mathscr H)$ to be this set of heights. In symbols, we have \[X_r(\mathscr H)=\{w_j:(j,w_j)\in Q_r(\mathscr H)\}.\] 

For example, suppose $\mathscr H$ is the valid hook configuration of $21123214567$ shown in Figure \ref{Fig6}. Referring to the induced coloring shown in Figure \ref{Fig7}, we find that 
\[Q_0(\mathscr H)=\{(1,2),(5,3),(12,7)\},\hspace{0.3cm}Q_1(\mathscr H)=\{(2,1)\},\hspace{0.3cm}Q_2(\mathscr H)=\{(3,1)\},\hspace{0.3cm}Q_3(\mathscr H)=\{(4,2)\},\] \[Q_4(\mathscr H)=\{(6,2),(9,4),(10,5)\},\hspace{0.3cm}Q_5(\mathscr H)=\{(7,1)\},\hspace{0.3cm}Q_6(\mathscr H)=\{(8,2)\},\hspace{0.3cm}Q_7(\mathscr H)=\{(11,6)\}\]
and 
\[X_0(\mathscr H)=\{2,3,7\},\hspace{0.3cm}X_1(\mathscr H)=\{1\},\hspace{0.3cm}X_2(\mathscr H)=\{1\},\hspace{0.3cm}X_3(\mathscr H)=\{2\},\] \[X_4(\mathscr H)=\{2,4,5\},\hspace{0.3cm}X_5(\mathscr H)=\{1\},\hspace{0.3cm}X_6(\mathscr H)=\{2\},\hspace{0.3cm}X_7(\mathscr H)=\{6\}.\]

The Catalan numbers appear in Theorem \ref{mainVHCthm} because $C_n$ is the number of (unlabeled) binary plane trees with $n$ nodes. Equivalently, if $X$ is a set of positive integers with $|X|=n$, then $C_n$ is the number of decreasing binary plane trees on $X$ whose postorder readings are in increasing order (since there is a unique way to add labels to each unlabeled binary plane tree so that its postorder is increasing).
If $\mathscr H=(H_1,\ldots,H_k)$ is a valid hook configuration that induces the valid composition $q^{\mathscr H}=(q_0,\ldots,q_k)$, then $q_r=|Q_r(\mathscr H)|=|X_r(\mathscr H)|$. We say that a tuple $\mathscr T=(T_0,\ldots,T_k)$ \emph{spawns from $\mathscr H$} if, for each $r\in\{0,\ldots,k\}$, $T_r$ is a decreasing binary plane tree on $X_r(\mathscr H)$ whose postorder reading is in increasing order. There are exactly $C_{q^{\mathscr H}}$ tuples that spawn from $\mathscr H$.\footnote{In general, the number of trees of a certain type with a prescribed postorder reading is given by a sum of products of numbers that count certain unlabeled plane trees, where the sum ranges over a specific set of valid hook configurations. This is explained in the context of permutations in \cite{Defant}.}  We are finally equipped to prove Theorem \ref{mainVHCthm}.

\begin{proof}[Proof of Theorem \ref{mainVHCthm}] Fix a word $w$. The proof that $|\fast^{-1}(w)|=\sum_{\mathscr H\in\mathcal H_{\mathcal R}(w)}C_{q^{\mathscr H}}$ consists of three steps. The first step is the description of an algorithm that produces a tree in $\mathcal R$ from a pair $(\mathscr H,\mathscr T)$, where $\mathscr H\in\mathcal H_{\mathcal R}(w)$ and $\mathscr T$ is a tuple that spawns from $\mathscr H$. The second step is a demonstration that each tree produced from this algorithm in fact has postorder $w$. The third step is a demonstration that every tree in $\mathcal R$ with postorder $w$ arises in a unique way from this algorithm. The second and third steps are virtually identical to the proofs of Proposition 3.1 and Theorem 3.1 in \cite{Defant}, so we will omit them here. To show that $|\slow^{-1}(w)|=\sum_{\mathscr H\in\mathcal H_{\mathcal L}(w)}C_{q^{\mathscr H}}$, we will describe a slightly different algorithm that produces a tree in $\mathcal L$ from a pair $(\mathscr H,\mathscr T)$, where $\mathscr H\in\mathcal H_{\mathcal L}(w)$ and $\mathscr T$ is a tuple that spawns from $\mathscr H$. As before, we will not go through the details of proving that this algorithm produces a tree with postorder $w$ and that every tree in $\mathcal L$ with postorder $w$ arises uniquely in this fashion. 

For the first algorithm, suppose we are given a word $w=w_1\cdots w_m$ and a pair $(\mathscr H,\mathscr T)$ such that $\mathscr H=(H_1,\ldots,H_k)\in\mathcal H_{\mathcal R}(w)$ and $\mathscr T=(T_0,\ldots,T_k)$ is a tuple that spawns from $\mathscr H$. We will construct a sequence of trees $\tau_m,\tau_{m-1},\ldots,\tau_1$ (in this order), and the tree $\tau_1$ will be the output of the algorithm. The tree $\tau_m$ consists of a single root vertex with label $w_m$. At each step, we produce $\tau_\ell$ by adding a leaf vertex $y_\ell$ with label $w_{\ell}$ to the tree $\tau_{\ell+1}$. There are two cases to consider when describing how to attach $y_\ell$ to $\tau_{\ell+1}$. 

\textbf{Case 1:} Suppose $(\ell,w_\ell)$ is the southwest endpoint of a hook $H_t$. Note that $H_t$ is the only hook with southwest endpoint $(\ell,w_\ell)$. Let $(j,w_j)$ be the northeast endpoint of $H_t$. If $y_j$ has no children in $\tau_{\ell+1}$, make $y_\ell$ a right child of $y_j$. If $y_j$ already has a right child in $\tau_{\ell+1}$, make $y_\ell$ a left child of $y_j$. 

\textbf{Case 2:} Suppose $(\ell,w_\ell)$ is not the southwest endpoint of any hook in $\mathscr H$. Let $u$ be the largest element of $\{1,\ldots,\ell\}$ such that $(u,w_u)$ is given the same color as $(\ell+1,w_{\ell+1})$ in the coloring induced by $\mathscr H$. In other words, if $r$ is the unique integer such that $(\ell+1,w_{\ell+1})\in Q_r(\mathscr H)$, then $u$ is the largest element of $\{1,\ldots,\ell\}$ such that $(u,w_u)\in Q_r(\mathscr H)$. One can show that $u$ exists.\footnote{In fact, $u$ is the smallest integer such that $(u,w_u)$ is connected to $(\ell,w_\ell)$ via a connected sequence of hooks in $\mathscr H$.} Furthermore, Lemma \ref{lem:colorsincrease} tells us that $w_u<w_{\ell+1}$. This implies that $w_u$ is not the largest element of $X_r(\mathscr H)$, so it has a parent in the tree $T_r$. Let $(v,w_v)\in Q_r(\mathscr H)$ be the point such that $w_v$ is the parent of $w_u$ in $T_r$. We know that $w_u<w_v$ because $T_r$ is a decreasing binary plane tree on $X_r(\mathscr H)$. Lemma \ref{lem:colorsincrease} guarantees that $u<v$, so our choice of $u$ forces $v\geq \ell+1$. This means that $y_v$ is a vertex in $\tau_{\ell+1}$. If $y_v$ has no children in $\tau_{\ell+1}$, make $y_\ell$ a right child of $y_v$. If $y_v$ already has a right child in $\tau_{\ell+1}$, make $y_\ell$ a left child of $y_v$.  

Case $1$ is really telling us that, roughly speaking, the hooks in $\mathscr H$ turn into some (but not all) of the edges in our tree. We now wish to show that the tree $\tau_1$ that we produce at the end of the algorithm is actually a weakly decreasing plane tree. If $(\ell,w_\ell)$ is the southwest endpoint of a hook, then it is clear from the definition of a hook that $y_\ell$ is attached as a child of a vertex whose label is greater than or equal to $w_\ell$. Suppose $(\ell,w_\ell)$ is not the southwest endpoint of a hook. Let $u$ and $v$ be as in the description of Case $2$ above. We need to show that $w_\ell\leq w_v$; we will actually show the stronger statement that $w_\ell<w_v$. Indeed, if $w_\ell\geq w_v$, then we can let $\ell'$ be the largest integer such that $\ell'<v$ and $w_{\ell'}\geq w_v$. As in the proof of Lemma \ref{lem:colorsincrease}, $(\ell',w_\ell')$ must be a descent top of $w$, so it is the southwest endpoint of a hook $H$. The point $(v,w_v)$ must be given the same color as either $H$ or a hook whose southwest endpoint is to the right of $(\ell',w_{\ell'})$. The point $(u,w_u)$ cannot be given this same color, which contradicts the fact that $(u,w_u)$ and $(v,w_v)$ have the same color. 

Next, we establish that we can actually perform every step of the above algorithm.  It suffices to show that we never reach a stage at which we try to attach a leaf as a child of a vertex that already has two children.  Choose an arbitrary point $(j,w_j)$, and let $t$ be the unique integer such that $(j,w_j)\in Q_t(\mathscr H)$.

Suppose $(j,w_j)$ is the northeast endpoint of a hook. We have assumed that $\mathscr H\in\mathcal H_{\mathcal R}(w)$, so $\mathscr H$ is a binary valid hook configuration. This means that there are at most two hooks with northeast endpoint $(j,w_j)$, so at most two vertices can be added as children of $y_j$ via Case 1. Remark \ref{Rem1} tells us that there is no point with the same color as $(j,w_j)$, so the tree $T_t$ consists of a single vertex. This implies that no vertex can be added as a child of $y_j$ via Case $2$. 

Next, suppose $(j,w_j)$ is not the northeast endpoint of a hook. Clearly, no vertex can be added as a child of $y_j$ via Case $1$. Because $T_t$ is a \emph{binary} plane tree, $w_j$ has at most two children in $T_t$. Each child of $w_j$ in $T_t$ can give rise to at most a single child of $w_j$ via Case $2$, and it follows that at most two vertices can be added as children of $y_j$ via Case $2$.

Finally, we need to discuss why the tree $\tau_1$ is in $\mathcal R$. The above argument shows that $\tau_1$ is a weakly decreasing binary plane tree, so we must explain why no vertex in $\tau_1$ has the same label as its left child. This is where we use the fact that every horizontal hook in $\mathscr H$ is small (by the definition of $\mathcal H_{\mathcal R}(w)$). Indeed, suppose $y_a$ is a vertex in $\tau_1$ with a left child $y_b$. If $y_b$ was attached to $y_a$ via Case $1$, then there must be a hook in $\mathscr H$ with southwest endpoint $(a,w_a)$ and northeast endpoint $(b,w_b)$. If this hook were small, then $y_b$ would have been added as a right child of $y_a$ instead of a left child. This means that the hook cannot be small, so it cannot be horizontal. Hence, $w_a<w_b$. On the other hand, if $y_b$ was added to $y_a$ via Case $2$, then the paragraph immediately following the description of Case 2 makes it clear that $w_a<w_b$.    

It now remains to describe the algorithm that produces a tree in $\mathcal L$ from a pair $(\mathscr H,\mathscr T)$, where $\mathscr H\in\mathcal H_{\mathcal L}(w)$ and $\mathscr T$ is a tuple that spawns from $\mathscr H$. The first part of the algorithm runs exactly as the previous algorithm. More precisely, we produce trees $\tau_m,\tau_{m-1},\ldots,\tau_1$ using the exact same procedure as before. We then modify the tree $\tau_1$ to create a tree $\tau_1'\in\mathcal L$. 

The same arguments as above show that $\tau_1$ is a weakly decreasing binary plane tree. Suppose $y_c$ is a right child of $y_d$ in $\tau_1$ and $w_c=w_d$. We claim that $y_c$ is the only child of $y_d$ in $\tau_1$. This means that we can simply ``swing" $y_c$ (along with its subtree) to the left so that it becomes a left child of $y_d$. Once we swing all of the right children that have the same labels as their parents, we will be left with our desired tree $\tau_1'\in\mathcal L$. 

It remains to prove the claim that $y_c$ is the only child of $y_d$ in $\tau_1$ whenever $y_c=y_d$. We have seen that $y_c$ could not have been attached as a child of $y_d$ via Case $2$ (since if it were, we would have $w_c<w_d$). Therefore, $(c,w_c)$ is the southwest endpoint of a hook $H$ with northeast endpoint $(d,w_d)$. According to condition 3 in the definition of a valid hook configuration, $(d,w_d)$ is the northeast endpoint of a small hook. This small hook has southwest endpoint $(d-1,w_{d-1})$. It follows from the description of Case $1$ in the above algorithm that $y_{d-1}$ is the right child of $y_d$ in $\tau_1$. This forces $c=d-1$, so $H$ is a small horizontal hook. Since $\mathscr H\in\mathcal H_{\mathcal L}(w)$, we know that $H$ is the only hook with northeast endpoint $(d,w_d)$. Accordingly, no vertex other than $y_c$ could have been attached as a child of $y_d$ via Case $1$. Because $(d,w_d)$ is the northeast endpoint of a hook, it follows from Remark \ref{Rem1} that no vertex could have been added as a child of $y_d$ via Case $2$. This proves the claim.
\end{proof}

\section{Fertility Numbers}
As an immediate application of the theory developed in the previous section, we prove a result 
concerning what we call $\fast$-fertility numbers and $\slow$-fertility numbers. Recall that West defined the fertility of a permutation $\pi$ to be $|s^{-1}(\pi)|$. In \cite{Defant4}, the first author defined a \emph{fertility number} to be a nonnegative integer $f$ such that there exists a permutation with fertility $f$. Among other things, he showed that $3,7,11,15,19,$ and $23$ are not fertility numbers, and he has conjectured that infinitely many positive integers are not fertility numbers. By analogy, we define a \emph{$\fast$-fertility number} to be a nonnegative integer $f$ such that there exists a word $w$ with $|\fast^{-1}(w)|=f$. We define \emph{$\slow$-fertility numbers} similarly. It turns out that $\fast$-fertility and $\slow$-fertility numbers are much less mysterious than ordinary fertility numbers. 

\begin{theorem}\label{thm:fertilitynumbers}
For every nonnegative integer $f$, there exists a word $w$ such that $|\fast^{-1}(w)|=|\slow^{-1}(w)|=f$. 
\end{theorem}

\begin{proof}
It is clear that the word $21$ has fertility $0$ under both $\fast$ and $\slow$ and that the word $1$ has fertility  $1$ under each map.  In \cite{Defant4}, it is shown that the permutation $\xi_m=m(m-1)\cdots 321(m+1)(m+2)\cdots (2m)$ has fertility $2m$ for every integer $m \geq 1$. Since $\fast$ and $\slow$ both restrict to the map $s$ on the set of permutations, this tells us that \[|\fast^{-1}(\xi_m)|=|\slow^{-1}(\xi_m)|=|s^{-1}(\xi_m)|=2m.\] For each integer $m \geq 1$, let $\xi_m'=m(m-1)\cdots 3211(m+1)(m+2)\cdots (2m)$ be the word obtained by inserting an additional $1$ between the $2$ and the $1$ in $\xi_m$. We will finish the proof by showing that \[|\fast^{-1}(\xi_m')|=|\slow^{-1}(\xi_m')|=2m+1.\] The reader may find it helpful to refer to Figure \ref{Fig8}, which shows the valid hook configurations of $\xi_3'=3211456$ and their induced colorings. 

\begin{figure}[t]
\begin{center} 
\includegraphics[width=\linewidth]{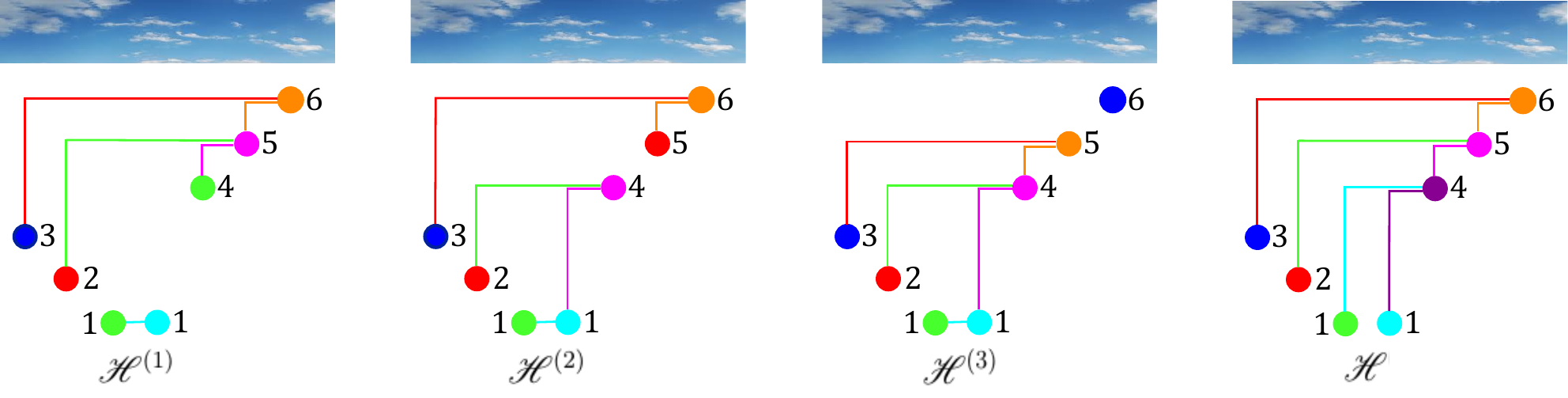}
\end{center}
\caption{The valid hook configurations of $\xi_3'=3211456$, which are described in the proof of Theorem \ref{thm:fertilitynumbers}.}\label{Fig8}
\end{figure}

Note that the only possible horizontal hook in a valid hook configuration of $\xi_m'$ is the hook with southwest endpoint $(m,1)$ and northeast endpoint $(m+1,1)$. It follows that $\mathcal H_{\mathcal R}(\xi_m')$ and $\mathcal H_{\mathcal L}(\xi_m')$ are both equal to the set of all binary valid hook configurations of $\xi_m'$. Let us choose such a binary valid hook configuration. The descent tops of $\xi_m'$ are precisely the points of the form $(i,m+1-i)$ for $i\in\{1,\ldots,m\}$. Let $H_i$ denote the hook with southwest endpoint $(i,m+1-i)$. 

Let us first suppose that $H_m$ has northeast endpoint $(m+1,1)$. The northeast endpoints of $H_1,\ldots,H_{m-1}$ form an $(m-1)$-element subset of $\{(m+2,m+1),(m+3,m+2),\ldots,(2m+1,2m)\}$. Of course, this subset is uniquely determined by choosing the positive integer $j$ such that $(m+1+j,m+j)$ is not in the subset. Once this element is chosen, the hooks $H_1,\ldots,H_{m-1}$ are uniquely determined. For each $i\in\{1,\ldots,m-1\}$, we must add an additional small hook that has the same northeast endpoint as $H_i$. This produces a binary valid hook configuration $\mathscr H^{(j)}$. In the coloring of the plot of $\xi_m'$ induced by $\mathscr H^{(j)}$, all of the points are given distinct colors, with one exception: the points $(m+1-j,j)$ and $(m+1+j,m+j)$ are both given the same color as $H_{m-j}$ (where $H_0$ denotes the sky). The valid composition induced from this valid hook configuration is $q^{\mathscr H^{(j)}}=(1,\ldots,1,2,1\ldots,1)$, where the $2$ is in the $(m+1-j)$-th position. Since $C_{(1,1,\ldots,1,2,1\ldots,1)}=2$, the valid hook configuration $\mathscr H^{(j)}$ contributes a $2$ to each of the sums in Theorem \ref{mainVHCthm}. 

Second, suppose that we have a binary valid hook configuration of $\xi_m'$ in which $H_m$ does not have northeast endpoint $(m+1,1)$. This forces $H_i$ to have northeast endpoint $(2m+2-i,2m+1-i)$ for every $i\in\{1,\ldots,m-1\}$. For each $i\in\{1,\ldots,m\}$, we must add an additional small hook that has the same northeast endpoint as $H_i$. This produces a binary valid hook configuration $\mathscr H ^{+}$. In the coloring of the plot of $\xi_m'$ induced by $\mathscr H^+$, all of the points are given distinct colors. Thus, $q^{\mathscr H}=(1,1,\ldots,1)$. The valid hook configuration $\mathscr H$ contributes $C_{(1,1,\ldots,1)}=1$ to each of the sums in Theorem \ref{mainVHCthm}. In summary, 
\[|\fast^{-1}(\xi_m')|=|\slow^{-1}(\xi_m')|=\sum_{j=1}^m C_{q^{\mathscr H^{(j)}}}+C_{q^{\mathscr H ^+}}=2m+1. \qedhere\]
\end{proof}

\section{Sortable Words}

The $t$-stack-sortable permutations mentioned in the introduction have received a large amount of attention \cite{Bona, BonaSurvey, Defant2,West, Consumption}.  We define a $t$-$\fast$-sortable word to be a word $w$ such that $\fast^t(w)$ is an identity word. In other words, it is a word $w$ such that $\langle w\rangle_{\fast}\leq t$. We define $t$-$\slow$-sortable words similarly.  Our goal in this section is to investigate the $1$-$\fast$-sortable words and $1$-$\slow$-sortable words. For brevity, we call these words \emph{$\fast$-sortable} and \emph{$\slow$-sortable}, respectively. 

Recall that a permutation is sortable if and only if it avoids the pattern $231$.  We begin with the corresponding characterization for sortable words.

\begin{proposition}\label{prop:sortablechar}
A word is $\fast$-sortable if and only if it avoids the pattern $231$.  A word is $\slow$-sortable if and only if it avoids the patterns $231$ and $221$.
\end{proposition}
\begin{proof}
We prove the contrapositive of each statement.  Let $w=w_1w_2\cdots w_m$.  First, suppose $w$ contains the pattern $231$, i.e., there exist $1\leq a<b<c\leq m$ such that $w_c<w_a<w_b$.  Consider the action of $\fast$ on $w$.  Because $w_a<w_b$, it is clear that $w_b$ will force $w_a$ to pop out of the stack if it has not already left the stack, and this occurs before $w_c$ even enters the stack.  Hence, $w_a$ precedes $w_c$ in $\fast(w)$, which implies that $\fast(w) \neq \Id$.  Second, suppose $\fast(w)=w_1'w_2'\cdots w_m' \neq \Id$.  Then there exist $1\leq d <e \leq m-1$ such that $w_d'>w_e'$.  (We have the restriction $e \leq m-1$ because no letter is larger than $w_m'$.)  The letter $w_d'$ must have exited the stack before $w_e'$ could even enter it. Let $w_f'$ be the letter that forces $w_d'$ to pop out of the stack. We must have $w_f'>w_d'>w_e'$. Furthermore, these letters must appear in the order $w_d',w_f',w_e'$ in $w$, which means that these three letters form a $231$ pattern in $w$.  This establishes the first statement. 

The proof of the second statement proceeds in a similar manner.  The only difference is that we replace the inequalities $w_a<w_b$ and $w_d'<w_f'$ by $w_a\leq w_b$ and $w_d'\leq w_f'$. 
\end{proof}

This proposition yields an immediate comparison between $|\fast^{-1}(\Id_{\textbf{c}})|$ and $|\slow^{-1}(\Id_{\textbf{c}})|$ for various vectors $\textbf{c}=(c_1, \ldots, c_n)$; the result holds particular interest in light of the discussion of Section $2$.  We remark that an alternative proof of this fact can be obtained by considering the map discussed in Section $3$ between $\mathcal{R}$ and $\mathcal{L}$ that turns right children into left children when these children equal their parents.

\begin{corollary}
For any $\textbf{c}=(c_1, \dots, c_n)$, where $c_1, \dots, c_n$ are positive integers, we have $$|\fast^{-1}(\Id_{\textbf{c}})| \geq |\slow^{-1}(\Id_{\textbf{c}})|.$$  Moreover, equality holds exactly when $c_i=1$ for all $i >1$.
\end{corollary}
\begin{proof}
Fix some $\textbf{c}=(c_1, \dots, c_n)$.  Since any word $w \in \slow^{-1}(\Id_{\textbf{c}})$ avoids the patterns $231$ and $221$, it is also in $\fast^{-1}(\Id_{\textbf{c}})$.  Hence, $\slow^{-1}(\Id_{\textbf{c}}) \subseteq \fast^{-1}(\Id_{\textbf{c}})$, which establishes the inequality.

Now, suppose $c_i=1$ for all $i>1$.  Then it is impossible for any word $w \in \mathcal{W}_{\textbf{c}}$ to contain the pattern $221$, so the conditions for $w \in \mathcal{W}_{\textbf{c}}$ being in $\fast^{-1}(\Id_{\textbf{c}})$ and $\slow^{-1}(\Id_{\textbf{c}})$ are equivalent. We can conclude that $|\fast^{-1}(\Id_{\textbf{c}})|=|\slow^{-1}(\Id_{\textbf{c}})|$ in this case.  Finally, suppose that $c_i\geq 2$ for some $i>1$.  Consider the word $w \in \mathcal{W}_{\textbf{c}}$ that is obtained from $\Id_{\textbf{c}}$ by moving all of the $i-1$'s to the right of the $i$'s.  Since $w$ contains the pattern $221$ but not the pattern $231$, it is in $\fast^{-1}(\Id_{\textbf{c}})$ but not in $\slow^{-1}(\Id_{\textbf{c}})$.  Hence, $|\fast^{-1}(\Id_{\textbf{c}})| > |\slow^{-1}(\Id_{\textbf{c}})|$ is strict in this case.
\end{proof}

We devote the remainder of this section to enumerating the $\fast$-sortable and $\slow$-sortable words. According to Proposition \ref{prop:sortablechar}, this is equivalent to the more classical problem of enumerating the words that avoid $231$ and the words that avoid both $231$ and $221$. 

Let us focus first on $\fast$. In its most general form, our problem is find a formula, depending on $c_1,\ldots,c_n$, for the number of $\fast$-sortable words in $\mathcal W_{(c_1,\ldots,c_n)}$. An explicit formula seems unattainable in this level of generality, but we can at least obtain a recurrence. In fact, this has already been done. Because of Proposition \ref{prop:sortablechar}, the following theorem is equivalent to Lemma 3 in \cite{Atkinson}.

\begin{theorem}[\cite{Atkinson}]\label{thm:harerecurrence}
For nonnegative integers $c_1, \dots c_n$, let $\mathcal M(c_1,\ldots,c_n)=|\fast^{-1}(\Id_{(c_1,\ldots,c_n)})|$ denote the number of $\fast$-sortable words in $\mathcal W_{(c_1,\ldots,c_n)}$. We have $\mathcal M(c_1)=1$ for all choices of $c_1$. For $n\geq 2$, we have \[\mathcal M(c_1,\ldots,c_n)=\begin{cases} \mathcal M(c_1+c_2,c_3,\ldots,c_n)+\sum_{r=1}^{c_1}\mathcal M(r,c_2-1,c_3,\ldots,c_n) & \mbox{if } c_2\geq 1 \\ \mathcal M(c_1,c_3,\ldots,c_n) & \mbox{if } c_2=0. \end{cases}\]
\end{theorem}

The authors of \cite{Atkinson} used Theorem \ref{thm:harerecurrence} to find an explicit formula for the generating function of $\mathcal M(c_1,\ldots,c_n)$. Specifically, given variables $x_1,x_2,\ldots$, let $y_i=x_i(1-x_i)$. Let $A(z_1,\ldots,z_m)=\prod_{1\leq i<j\leq m}(z_i-z_j)$. The following theorem is Theorem $3$ in \cite{Atkinson}. 

\begin{theorem}[\cite{Atkinson}]\label{thm:haregenfunc}
In the above notation, we have \[\sum_{a_1,\ldots,a_n\geq 0}\mathcal M(a_1,\ldots,a_n)x_1^{a_1}\cdots x_n^{a_n}=-\frac{\sum_{i=1}^n(-1)^ix_iy_i^{n-2}A(y_1,\ldots,y_{i-1}y_{i+1}\ldots,y_n)}{A(y_1,\ldots,y_n)}.\]
\end{theorem}

As a corollary of Theorem \ref{thm:haregenfunc}, the authors of \cite{Atkinson} proved the surprising fact that $\mathcal M(c_1,\ldots,c_n)$ is a symmetric function of the arguments $c_1,\ldots,c_n$. That is, for any permutation $\sigma_1\cdots\sigma_n\in S_n$, \[\mathcal M(c_1,\ldots,c_n)=\mathcal M(c_{\sigma_1},\ldots,c_{\sigma_n}).\] 

We now turn our attention to deriving a recurrence relation for the $\slow$-fertility of $\Id_{(c_1,\ldots,c_n)}$.  Let $\mathcal{N}(c_1, \ldots, c_n)=|\slow^{-1}(\Id_{\mathbf{c}})|$ denote the number of $\slow$-sortable words in $\mathcal W_{(c_1,\ldots,c_n)}$. Equivalently, $\mathcal N(c_1,\ldots,c_n)$ is the number of words in $\mathcal W_{(c_1,\ldots,c_n)}$ that avoid the patterns $231$ and $221$. The following theorem reveals $\mathcal{N}(c_1, \dots, c_n)$ not to depend on the value of $c_n$. 

\begin{theorem}\label{thm:slowfertility}
For any positive integer $c_1$, we have $\mathcal{N}(c_1)=1$.  Moreover, for $n \geq 2$ and any positive integers $c_1, \dots, c_n$, we have
\begin{align*}
\mathcal{N}(c_1, \dots, c_n) &=2\mathcal{N}(c_1, \dots, c_{n-1})+\sum_{i=1}^{n-2}\mathcal{N}(c_1, \dots, c_i)\mathcal{N}(c_{i+1}, \dots, c_{n-1})\\
 &\quad +\sum_{i=1}^{n-1}\sum_{k=1}^{c_i-1}\mathcal{N}(c_1, \dots, c_{i-1},k)\mathcal{N}(c_i-k, c_{i+1}, \dots, c_{n-1}).
\end{align*}
\end{theorem}

\begin{proof}
The $n=1$ case is easy: $\mathcal{W}_{(c_1)}$ consists of only the identity word, which is clearly $\slow$-sortable, so $\mathcal{N}(c_1)=1$.

Now, consider $n \geq 2$.  Consider a $\slow$-sortable word $w\in\mathcal W_{(c_1,\ldots,c_n)}$. Since $w$ avoids the pattern $221$, all but one of the $n$'s must be at the very end of $w$, i.e., $w=AnBnn\cdots n$ (where there are $c_n-1$ $n$'s appearing at the end) for some (possibly empty) words $A$ and $B$ that do not contain the letter $n$.  We can now compute $$\slow(AnBnn\cdots n)=\slow(A) \slow(B) nnn \cdots n,$$ and this sorted word is the identity exactly when both $A$ and $B$ are $\slow$-sortable and no letter of $A$ is larger than a letter of $B$.  Note that $AB \in \mathcal{W}_{(c_1, \dots, c_{n-1})}$.

If $A$ is empty, then $B \in \mathcal{W}_{(c_1, \dots, c_{n-1})}$, and by definition there are $\mathcal{N}(c_1, \dots, c_{n-1})$ possible choices for $B$.  Similarly, if $B$ is empty, then there are $\mathcal{N}(c_1, \dots, c_{n-1})$ possible choices for $A$.  This pair of possibilities gives the first term in the recurrence relation.

Now, suppose both $A$ and $B$ are nonempty and there is no letter value that appears in both $A$ and $B$.  Then there exists some $1\leq i \leq n-2$ such that $A \in \mathcal{W}_{(c_1, \dots, c_i)}$ and $B \in \mathcal{W}_{(c_{i+1}, \dots, c_{n-1})}$.  In this case, there are $\mathcal{N}(c_1, \dots, c_i)\mathcal{N}(c_{i+1}, \dots, c_{n-1})$ such pairs of sortable words $(A,B)$.  Summing over $i$ gives the second term in the recurrence relation.

Finally, consider the case where there is some value $1\leq i \leq n-1$ that appears in both $A$ and $B$.  Then there exists $1\leq k \leq c_i-1$ such that $A$ contains $k$ $i$'s and $B$ contains $c_i-k$ $i$'s.  Hence, we have $A \in \mathcal{W}_{(c_1, \dots, c_{i-1}, k)}$ and $B \in \mathcal{W}_{(c_i-k, c_{i+1}, \dots, c_{n-1})}$.  As above, there are $\mathcal{N}(c_1, \dots, c_{i-1}, k)\mathcal{N}(c_i-k, c_{i+1}, \dots, c_{n-1})$ such pairs of sortable words $(A,B)$.  Summing over $i$ and $k$ gives the third term in the recurrence relation.  This exhausts all possibilities.
\end{proof}
Table \ref{tab:formula} gives the formulas for $\mathcal{N}(c_1, \dots, c_n)$ and $\mathcal{M}(c_1, \dots, c_n)$ for small values of $n$.  It is not difficult to see that $\mathcal{N}(c_1, \dots, c_n)$ grows as $c_i^{n-i-1}$ in each $i$ (e.g., $\mathcal{N}(c_1, c_2, c_3)$ grows quadratically in $c_1$ and linearly in $c_2$).

\begin{table}[h]
\begin{tabular}{|c|c|}
\hline & \\
$\mathcal{N}(c_1)=1$ &$\mathcal{M}(c_1)=1$\\
\hline & \\
$\mathcal{N}(c_1, c_2)=c_1+1$ &$\mathcal{M}(c_1, c_2)=\binom{c_1+c_2}{c_1}$\\
\hline & \\
$\begin{aligned} \mathcal{N}(c_1, c_2, c_3) &=\frac{1}{2}c_1^2+c_1c_2+\frac{3}{2}c_1+c_2+1\\ &=\frac{1}{2}(c_1+1)(c_1+2c_2+2) \end{aligned}$ &$\begin{aligned} \mathcal{M}(c_1, c_2, c_3) &=2^{c_1+c_2+c_3}-\sum_{r=0}^{c_1-1} \binom{c_1+c_2+c_3}{r}\\ &-\sum_{r=0}^{c_2-1}  \binom{c_1+c_2+c_3}{r}-\sum_{r=0}^{c_3-1}  \binom{c_1+c_2+c_3}{r} \end{aligned}$\\
\hline
\end{tabular}
\vspace{3pt}
\caption{The formulas for $\mathcal{N}(c_1, \dots, c_n)$ and $\mathcal{M}(c_1, \dots, c_n)$ quickly become complicated as $n$ grows. The second column comes from \cite{Atkinson}.}
\label{tab:formula}
\end{table}
We remark that this type of argument yields the similar but more complicated recurrence relation for $\mathcal M(c_1,\ldots,c_n)$.  Here, a sortable word $w=A_1 n A_2 n \cdots n A_{k+1}$ no longer has to avoid the pattern $221$, so we lose the requirement that $A_i$ be empty for all $i \geq 3$.  Rather, all we need is that each $A_i$ be $\fast$-sortable and that no letter of $A_i$ be greater than a letter of $A_j$ for $i<j$.  This division of letters among the $A_i$'s corresponds to ``dividing'' the word $\Id_{(c_1, \dots, c_{n-1})}$ into $n+1$ (possibly empty) contiguous pieces.

Although the general formula in Theorem \ref{thm:slowfertility} looks complicated, it simplifies in some special cases.  In particular, we investigate the $\ell$-uniform (normalized) words. These are words in which each letter value that appears in the word appears exactly $\ell$ times, i.e., $\textbf{c}=(\ell,\ell,\dots, \ell)$. 

To count these words, we make use of \emph{generating trees}, an enumerative tool that was introduced in \cite{Chung} and studied extensively afterward \cite{Banderier, West3, West2}. To describe a generating tree of a class of combinatorial objects, we first specify a scheme by which each object of size $n$ can be uniquely generated from an object of size $n-1$. We then label each object with the number of objects it generates. The generating tree consists of an ``axiom" that specifies the labels of the objects of size $1$ along with a ``rule" that describes the labels of the objects generated by each object with a given label. For example, in the generating tree 
\[\text{Axiom: }(2)\qquad\text{Rule: }(1)\leadsto(2),\quad(2)\leadsto(1)(2),\] the axiom $(2)$ tells us that we begin with a single object of size $1$ that has label $2$. The rule $(1)\leadsto(2),\hspace{.15cm}(2)\leadsto(1)(2)$ tells us that each object of size $n-1$ with label $1$ generates a single object of size $n$ with label $2$, whereas each object of size $n-1$ with label $2$ generates one object of size $n$ with label $1$ and one object of size $n$ with label $2$. This example generating tree describes objects counted by the Fibonacci numbers. 

\begin{theorem}\label{uniformwords}
The number of $\ell$-uniform words on the alphabet $[n]$ that avoid the patterns $231$ and $221$ is \[\mathcal{N}(\underbrace{\ell, \ell, \dots, \ell}_{n})=\frac{1}{\ell n+1}\binom{(\ell +1)n}{n}.\] 
\end{theorem}

\begin{proof}
The authors of \cite{Banderier} show (their Example 9) that objects counted by the $(\ell+1)$-Catalan numbers $\frac{1}{\ell n+1}{(\ell+1)n\choose n}$ can be described via the generating tree 
\begin{equation}\label{Eq1}
\text{Axiom: }(\ell+1)\qquad\text{Rule: }(m)\leadsto(\ell+1)(\ell+2)\cdots(\ell+m)\quad\text{for every }m\in\mathbb N.
\end{equation} Fix some positive integer $\ell$, and let $\mathcal P_\ell(231,221)$ denote the set of all normalized $\ell$-uniform words that avoid the patterns $231$ and $221$; we will show that these words can be described using the generating tree in \eqref{Eq1}. 

Let us say that a word $w'\in\mathcal P_\ell(231,221)$ over the alphabet $[n]$ is generated from a word $w\in\mathcal P_\ell(231,221)$ over the alphabet $[n-1]$ if we can obtain $w'$ by inserting $\ell$ copies of the letter $n$ into spaces between the letters in $w$. For example, when $\ell=n=3$, the word $121122$ generates the words 
\begin{equation}\label{Eq2}
312112233, \quad 132112233, \quad 121132233, \quad 121123233, \quad 121122333.
\end{equation} Because $w'$ avoids $221$, the last $\ell-1$ letters of $w'$ all have value $n$. Therefore, $w'$ is determined by specifying $w$ along with the position $j$ of the first appearance of the letter $n$ in $w'$. In the above example, the possible positions $j$ where we could have placed the first appearance of the letter $3$ were $1,2,5,6,7$. In general, we can place the first appearance of $n$ into position $j$ if and only if $1\leq j\leq\ell(n-1)+1$ and there do not exist $\alpha,\beta$ such that $1\leq \alpha<j\leq\beta\leq\ell(n-1)$ and $w_\alpha>w_\beta$. Indeed, this follows from the requirement that the new word $w'$ avoids $231$. We label the word $w$ with the number of such positions $j$, or, equivalently, the number of words that $w$ generates. 

Suppose we are given the word $w\in\mathcal P_\ell(231,221)$ over the alphabet $[n-1]$. Let $m$ be the label of $w$, and let $j_1<\cdots<j_m$ be the positions where we can place the letter $n$ so that, after appending an additional $\ell-1$ copies of $n$ to the end of the word, we obtain a word $w'\in\mathcal P_\ell(231,221)$ over the alphabet $[n]$ that is generated by $w$. If we were to place the letter $n$ in the $j_r^\text{th}$ position between letters of $w$ and then append an additional $\ell-1$ copies of $n$ to the end, we would obtain a word $w'$ with label $\ell+r$. Indeed, the words generated by $w'$ can be formed by inserting the letter $n+1$ into one of the positions $j_1,\ldots,j_r,\ell(n-1)+2,\ldots,\ell n+1$ between letters in $w'$ and appending $\ell-1$ copies of $n+1$ to the end. Therefore, $w$ (which has label $m$) generates words with labels $\ell+1,\ell+2,\ldots,\ell+m$. For example, the word $121122$ has label $5$ and generates the words in \eqref{Eq2}, which have labels $4,5,6,7,8$, respectively. This is precisely the rule in the generating tree in \eqref{Eq1}. Of course, the only word in $\mathcal P_\ell(231,221)$ over the alphabet $[1]$ is $11\cdots 1$ (of length $\ell$). This word has label $\ell+1$, which yields the axiom of the generating tree in \eqref{Eq1}.      
\end{proof}

\section{Concluding Remarks and Further Directions}

The introduction of the maps $\fast$ and $\slow$ leads to a variety of interesting problems, which we list in this section. 

Theorem \ref{thm:tortoisebeatsharebyalot} tells us that for each positive integer $k$, there is a word $\eta_{k+2}$ of length $2k+5$ with the property that $\langle \eta_{k+2}\rangle_{\fast}-\langle \eta_{k+2}\rangle_{\slow}=k$. More precisely, $\langle \eta_{k+2}\rangle_{\fast}=2k+2$ and $\langle \eta_{k+2}\rangle_{\slow}=k+2$. We suspect that $\eta_{k+2}$ is minimal among such words in the sense of the following conjectures. 

\begin{conjecture}\label{Conj2}
If $w$ is a word of length $m$, then \[\langle w\rangle_{\fast}-\langle w\rangle_{\slow}\leq\frac{m-5}{2}.\]  
\end{conjecture}

\begin{conjecture}\label{Conj3}
For every word $w$, we have \[\langle w\rangle_{\fast}\leq 2\langle w\rangle_{\slow}-2.\]
\end{conjecture}

After Theorem \ref{thm:tortoisebeatsharebyalot}, we defined an exceptional word to be a word $w$ such that $\langle w\rangle_{\fast}>\langle w\rangle_{\slow}$. We also let $\mathcal E_m$ denote the set of exceptional normalized words of length $m$. We have calculated that $|\mathcal E_m|=0$ for $m\leq 6$, $|\mathcal E_7|=4$, $|\mathcal E_8|=172$, and $|\mathcal E_9|=5001$. Let $\mathcal{NW}_m$ denote the set of normalized words of length $m$. We are interested in the ratios $|\mathcal E_m|/|\mathcal{NW}_m|$. These values for $m=7,8,9$ are (approximately) $0.000085$, $0.000315$, $0.000706$. This leads us to make the following conjecture. 

\begin{conjecture}\label{Conj1}
The numbers \[\frac{|\mathcal E_m|}{|\mathcal{NW}_m|}\] are increasing in $m$.
\end{conjecture}

If Conjecture \ref{Conj1} is true, then $\displaystyle\lim_{m\to\infty}|\mathcal E_m|/|\mathcal{NW}_m|$ exists. It would be very interesting to calculate (or at least estimate) this limit.  

Each element of $\mathcal E_8$ contains one of the words in $\mathcal E_7$ as a pattern. This suggests that it could be possible to find conditions based on pattern avoidance that are necessary and/or sufficient for a word to be exceptional.  

We saw in Section 4 that every nonnegative integer is a $\fast$-fertility number and a $\slow$-fertility number. In other words, if we define maps $\mathcal F_{\fast},\mathcal F_{\slow}:\mathcal W\to\mathbb N\cup\{0\}$ by $\mathcal F_{\fast}(w)=|\fast^{-1}(w)|$ and $\mathcal F_{\slow}(w)=|\slow^{-1}(w)|$, then \[\mathcal F_{\fast}(\mathcal W)=\mathcal F_{\slow}(\mathcal W)=\mathbb N\cup\{0\}.\]

Let $\mathcal P$ denote the set of all permutations. The first author has conjectured \cite{Defant4} that there are infinitely many positive integers that are not in the set $\mathcal F_{\fast}(\mathcal P)=\mathcal F_{\slow}(\mathcal P)$ (where these sets are identical because $\fast$, $\slow$, and $s$ all agree on permutations). It would be interesting to see if this phenomenon persists when we restrict attention to certain natural sets of words. For example, we have the following question. Recall that a $2$-uniform word is a word in which each letter that appears actually appears exactly twice.  

\begin{question}
What can we say about $\mathcal F_{\fast}(\mathcal P_2)$ and $\mathcal F_{\slow}(\mathcal P_2)$, where $\mathcal P_2$ denotes the set of all $2$-uniform words? 
\end{question}

Recall from the beginning of Section 5 that a word $w$ is $t$-$\fast$-sortable (respectively, $t$-$\slow$-sortable) if $\langle w\rangle_{\fast}\leq t$ (respectively, $\langle w\rangle_{\slow}\leq t$). We have not said anything about these families of words when $t\geq 2$. It would be interesting to investigate $t$-$\fast$-sortable words and $t$-$\slow$-sortable words in general. In the past, there has been a huge amount of interest in $2$-stack-sortable permutations \cite{Bevan,Bona,BonaSurvey,Cori,Dulucq2,Fang,Goulden,West,Consumption}. It is probably very difficult to obtain an explicit formula for the number of $2$-$\fast$-sortable words (or $2$-$\slow$-sortable words) in $\mathcal W_{\textbf{c}}$ for arbitrary vectors $\textbf c$, but deriving recurrences might be possible. Also, one might be able to prove more refined statements about specific choices of $\textbf c$, such as $(1,\underbrace{2,\ldots,2}_{n-1})$, $(\underbrace{1,\ldots,1}_{n-1},2)$, and $(\underbrace{2,\ldots,2}_{n})$. 

Finally, let us mention that the authors of \cite{Defant3} have found several interesting properties of \emph{uniquely sorted} permutations, which are permutations with fertility $1$. Let us say a word $w$ is \emph{uniquely $\fast$-sorted} if $|\fast^{-1}(w)|=1$ and \emph{uniquely $\slow$-sorted} if $|\slow^{-1}(w)|=1$. We propose the investigation of uniquely $\fast$-sorted words and uniquely $\slow$-sorted words as a potential area for future research.  

\section{Acknowledgments}
This research was conducted at the University of Minnesota Duluth REU and was supported by NSF/DMS grant 1650947 and NSA grant H98230-18-1-0010.  The authors wish to thank Joe Gallian for running the REU program and providing encouragement.  The first author was also supported by a Fannie and John Hertz Foundation Fellowship and an NSF Graduate Research Fellowship.

\end{document}